\documentclass[11pt]{amsart}
\usepackage[margin=1.25in]{geometry}

\usepackage[utf8]{inputenc}
\usepackage{hyperref}
\usepackage{todonotes}
\usepackage[shortlabels]{enumitem}
\usepackage{mathtools}
\usepackage{microtype}
\usepackage{ytableau}
\usepackage{amsthm,amsmath,amssymb}

\newtheorem{theorem}{Theorem}[section]
\newtheorem*{theorem*}{Theorem}
\newtheorem{lemma}[theorem]{Lemma}
\newtheorem*{lemma*}{Lemma}
\newtheorem{proposition}[theorem]{Proposition}
\newtheorem*{proposition*}{Proposition}
\newtheorem{conjecture}[theorem]{Conjecture}

\theoremstyle{definition}
\newtheorem{definition}[theorem]{Definition}
\newtheorem*{definition*}{Definition}
\newtheorem{example}[theorem]{Example}
\newtheorem*{example*}{Example}

\newtheorem*{defn*}{Definition}
\newtheorem{corollary}[theorem]{Corollary}
\newtheorem*{corollary*}{Corollary}

\theoremstyle{remark}
\newtheorem*{remark}{Remark}

\newcommand{\Mod}[1]{\ (\mathrm{mod}\ #1)}

\title[Staircase Encodings and Boundary Grids]
{Resolving Two Conjectures on Staircase Encodings and Boundary Grids of $132$ and $123$-avoiding permutations}
\author{Shyam Narayanan}
\address{Department of Mathematics, Harvard University, Cambridge, MA 02138}
\email[Shyam Narayanan]{shyamnarayanan@college.harvard.edu}

\begin{document}

\begin{abstract}
	This paper analyzes relations between pattern avoidance of certain permutations and graphs on staircase grids and boundary grids, and proves two conjectures posed by Bean, Tannock, and Ulfarsson (2015). More specifically, this paper enumerates a certain family of staircase encodings and proves that the downcore graph, a certain graph established on the boundary grid, is pure if and only if the permutation corresponding to the boundary grid avoids the classical patterns 123 and 2143.
\end{abstract}

\maketitle

\section{Introduction and Preliminaries}
The study of permutation pattern avoidance became a significant topic in combinatorics starting in 1968, when Donald Knuth \cite{Knuth1968} noted that stack-sortable permutations corresponded to avoiding the permutation pattern $231$ (which we will explain shortly). Common problems include counting the number of permutations in $S_n$ (the set of permutations of $[n] = \{1, \dots, n\}$) avoiding a certain pattern or set of patterns, or determining asymptotics of these numbers in terms of $n$, or establishing Wilf-equivalences, i.e. showing that the number of permutations in $S_n$ avoiding a certain pattern (or set of) equals the number of permutations avoiding a different pattern (or set of). In this paper, we do not address these questions, but look at some related combinatorial objects explored in \cite{BeanTannockUlfarsson2016} and prove their conjectures.

We recall the definitions of order-isomorphic and containing a classical pattern.
\begin{definition}
	We say that two sequences $(a_1, \dots, a_k)$ and $(b_1, \dots, b_k)$ are \emph{order-isomorphic} if for any $1 \le i,j \le k$, we have $a_i < a_j$ if and only if $b_i < b_j$. In this case, we write $a_1 \cdots a_k \sim b_1 \cdots b_k$.
\end{definition}
\begin{definition}
	Let $\sigma \in S_k$.
	A permutation $\pi \in S_n$ \emph{contains} $\sigma$ as a \emph{classical pattern}
	if there exists $1 \le i_1 < \dots < i_k \le n$ such that 
	\[ \pi_{i_1} \pi_{i_2} \cdots \pi_{i_k} \sim \sigma_1\sigma_2\cdots\sigma_k. \]
	Alternatively, if a permutation does not contain $\sigma$,
	it is said to \emph{avoid} $\sigma$.
\end{definition}
\begin{example}
	The permutation $641532$ avoids the classical pattern $123$.
\end{example}

We also define a key set of permutations.
\begin{definition}
	For each $n$ and a permutation $\sigma$, we define $Av_n(\sigma)$ as the set of permutations in $S_n$ that avoid $\sigma$. We also say for a permutation $\pi \in Av_n(\sigma),$ that $\pi$ avoids $\sigma$ or $\pi$ is $\sigma$-avoiding. We can generalize this to avoiding more than one permutation pattern, by defining $Av_n(\sigma_1, \dots, \sigma_k)$ as the set of permutations in $S_n$ that avoid all of the permutations $\sigma_1, \dots, \sigma_k$.
	
	For a single permutation $\sigma$, we define $Co_n(\sigma)$ as the set of permutations in $S_n$ that contain $\sigma,$ so $Co_n(\sigma) = S_n \backslash Av_n(\sigma).$
\end{definition}

\begin{definition}
	For a permutation $\pi = \pi_1\pi_2\cdots \pi_n,$ we define $\pi_i$ to be a \emph{left-to-right minimum} if $\pi_j > \pi_i$ for all $j < i$.
\end{definition}

	Analogously define a \emph{left-to-right maximum}, \emph{right-to-left minimum}, and \emph{right-to-left maximum}.

\begin{example}
	In the permutation $536241$, $5, 3, 2, 1$ are the left-to-right minima and $1, 4, 6$ are the right-to-left maxima.
\end{example}

	For the remainder of the paper, we look primarily either at $123$-avoiding or $132$-avoiding permutations, so we note that it is well known that both $123$-avoiding permutations \cite{MacMahon1916} and $132$-avoiding permutations \cite{Knuth1968} are enumerated by the Catalan numbers, i.e. $|Av_n(123)| = \frac{1}{2n+1} \cdot {2n \choose n} = |Av_n(132)|.$

	In Section 2, we look at staircase grids, combinatorial structures related to permutations which are especially interesting in relation to 132-avoiding patterns, and enumerate a certain family of staircase grids, proving Conjecture 3.7 in \cite{BeanTannockUlfarsson2016} and providing one more enumeration not conjectured by \cite{BeanTannockUlfarsson2016}. In Section 3, we look at boundary grids and downcore graphs, combinatorial structures relating to 123-avoiding permutations, and show that the downcore graph of a grid is pure if and only if the permutation it relates to is both 123-avoiding and 2143-avoiding. This proves conjecture 6.1 in \cite{BeanTannockUlfarsson2016}.
	
\section{First Conjecture}

Consider an arbitrary permutation $\pi$ with the set of left-to-right minima $\pi_{k_1}, \pi_{k_2}, \dots, \pi_{k_a}$.

\begin{definition}
Define the \emph{staircase grid} $B_a$ as a staircase-like grid with $a$ rows and $a$ columns such that the top row and right column have $a$ boxes, the second-to-top row and second-to-right column have $a-1$ boxes, and so on. By a \textit{box}, we mean a unit square. See Figure \ref{StaircaseEncoding} for an example.

More specifically, define the \emph{staircase encoding} of $\pi$ as the grid $B_a$ combined with a number in each box. For the box in the $i^{\text{th}}$ row from the top and the $j^{\text{th}}$ column from the bottom, fill in the number of elements in $\pi$ which are strictly between $\pi_{k_i}$ and $\pi_{k_{i-1}}$ as a value and between position $k_j$ and $k_{j+1}$ in the order of the permutation (assume $\pi_{k_0} = k_{a+1} = n+1$). For an example, see Figure \ref{StaircaseEncoding}.
\end{definition}

\begin{figure}
\begin{ytableau}
\none & 2 & 0 & 1 \\
\none & \none & 1 & 0 \\
\none & \none & \none & 1
\end{ytableau}

\caption{The staircase encoding of $\pi = 58634127$. The staircase grid $B_3$ is the same, without numbers. The figure has $3$ left-to-right minima: $5, 3, 1$. The top-left corner is $2$ since $8, 6$ are greater than $5$ but come between $5$ and $3$ in the permutation. The other values are calculated similarly.}
\label{StaircaseEncoding}
\end{figure}

It turns out that staircase encodings have an interesting relation with $132$-avoiding permutations.

\begin{proposition}
	No two $132$-avoiding permutations have the same staircase encoding. \cite{BeanTannockUlfarsson2016}
\end{proposition}

\begin{proof}
	To avoid $132$ subsequences, note that in each row and column of the staircase encoding, the elements corresponding to each row and each column must be in increasing order, or else a $132$ subsequence will occur because of the left-to-right minima. The conclusion clearly follows.
\end{proof}

However, not all staircase encodings correspond to a $132$-avoiding permutations. For a staircase grid, consider the following graph.

\begin{definition}
	The \emph{downcore graph} of a staircase grid of size $n$ (i.e. $n$ rows, $n$ columns) is the undirected graph with the vertex set $\{(i, j): 1 \le i \le j \le n\}$ where $(i, j)$ refers to the box at the $i^{\text{th}}$ row from top and $j^{\text{th}}$ column from left. An edge exists between $(i, j)$ and $(k, \ell)$ if
  
\begin{itemize}
	\item $i < k \le j < \ell$ or
  \item $k < i \le \ell < j$.
\end{itemize}

We can alternatively think of this as requiring $i<k$ and $j<\ell$ or vice versa, and that $(i, \ell)$ and $(k, j)$ are boxes in the staircase grid.
\end{definition}

	A staircase encoding corresponds to a $132$-avoiding permutation if and only if the nonzero boxes in the staircase grid form an independent set in the downcore \cite{BeanTannockUlfarsson2016}.
  
  As a result, it is of interest to determine the number of distinct independent sets of size $k$ in the downcore of size $n$. If we denote this value $I(n, k),$ Bean et al. proved the following:
  
\begin{theorem} \cite{BeanTannockUlfarsson2016} \label{BeanThm}
 The generating function $$F(x, y) = \sum\limits_{n, k \in \mathbb{Z}_{\ge 0}} I(n, k) x^n y^k$$ satisfies the functional equation $$F = 1 + x \cdot F + \frac{xy \cdot F^2}{1 - y \cdot (F-1)},$$ and this means $$I(n, k) = \frac{1}{n}\sum\limits_{j=0}^{n-1} {n \choose k-j}{n \choose j+1}{n-1+j \choose n-1}.$$
\end{theorem}

We can now consider for the $132$-avoiding permutations of length $\ell$, the ones that have a staircase encoding with $k$ nonzero boxes in the grid.

\begin{proposition} \label{BeanProp}
	If we denote the number of $132$-avoiding permutations of length $\ell$ with $k$ nonzero boxes in the staircase encoding as $J(\ell, k),$ then $$J(\ell, k) = \sum\limits_{n = 0}^{\ell} I(n, k) {\ell - n - 1 \choose k-1},$$ which can be seen using the stars and bars method \cite{BeanTannockUlfarsson2016}.
\end{proposition}

Bean, Tannock, and Ulfarsson then pose the following conjecture, which we in fact prove using Theorem \ref{BeanThm} and Proposition \ref{BeanProp}.

\begin{conjecture} \label{ConjEasy}\cite{BeanTannockUlfarsson2016}
	Let $C_i$ denote the $i^{\text{th}}$ Catalan number. Then, for fixed $\ell$, if we consider the largest $k$ such that $J(\ell, k) \neq 0,$ then
\begin{enumerate}
	\item if $\ell = 3i+2,$ then $J(\ell, k) = C_i,$ and \label{woy1}
	\item if $\ell = 3i+1,$ then $J(\ell, k) = \frac{3}{2} {2i \choose i}$. \label{woy2}
\end{enumerate} 
\end{conjecture}

After proving this conjecture, we will also determine $J(\ell, k)$ the largest $k$ such that $J(\ell, k) \neq 0$, when $\ell = 3i$.

We begin by proving the following.

\begin{proposition} \label{prop27}
	For all $\ell \ge 1,$ the largest $k$ such that $J(\ell, k) \neq 0$ is $\lfloor\frac{2\ell-1}{3}\rfloor.$
\end{proposition}

\begin{proof}
  For $J(\ell, k)$ to be positive, there must exist some $n$ such that $I(n, k)$ and ${\ell - n - 1 \choose k - 1}$ are positive. In other words, we must have that $\ell - n \ge k$ and that $$\sum\limits_{j = 0}^{n-1} {n \choose k-j}{n \choose j+1}{n-1+j \choose n-1}$$ are both positive for some $n$. But the sum being positive means that there is some $j$ such that $n \ge k - j$ and $n \ge j + 1$ as well as $\ell - n \ge k.$ Adding the first two equations plus twice the third tells us that $2\ell \ge 3k+1.$ Basic manipulation gives $k \le \frac{2\ell-1}{3}$.
  
  Specifically, if $\ell = 3i+2,$ then $k \le 2i+1$; if $\ell = 3i+1,$ then $k \le 2i$; and if $\ell = 3i,$ then $k \le 2i-1.$
\end{proof}

We next prove part \eqref{woy1} of Conjecture \ref{ConjEasy}.
\begin{theorem} \label{EasyPart1}
	We have that $J(3i+2, 2i+1)$ is positive and equals $C_i,$ the $i^{\text{th}}$ Catalan number.
\end{theorem}

\begin{proof}
	Note that if $I(n, k)$ is positive, there must exist some $j$ such that $n \ge k - j$ and $n \ge j+1.$ Thus, $2n \ge k+1,$ which means
\begin{align*}
    J(3i+2, 2i+1) &= \sum\limits_{n = i+1}^{3i+2} I(n, 2i+1){(3i+2) - n - 1 \choose (2i+1) - 1} \\
    &= \sum\limits_{n = i+1}^{3i+2} I(n, 2i+1){3i+1-n \choose 2i}.
\end{align*}
  But the summands are only positive if $3i+1 - n \ge 2i,$ or if $n \le i+1.$ Thus, this summation in fact equals 
\begin{align*}
	I(i+1, 2i+1) {2i \choose 2i} &= I(i+1, 2i+1) \\
    &= \frac{1}{i+1} \sum\limits_{j = 0}^{i} {i+1 \choose 2i+1-j} {i+1 \choose j+1} {i+j \choose i}.
\end{align*} 
    As noted earlier, the summands are positive only if $i+1 \ge 2i+1 - j$ and $i+1 \ge j+1,$ i.e. if $j \ge i$ and $i \ge j$, which gives $i = j.$ Thus, this becomes 
\[\frac{1}{i+1} {i+1 \choose i+1}{i+1 \choose i+1}{2i \choose i} = \frac{1}{i+1}{2i \choose i} = C_i. \qedhere\]
\end{proof}

Next, we prove part 2 of Conjecture \ref{ConjEasy}.

\begin{theorem} \label{EasyPart2}
	We have that $J(3i+1, 2i)$ is positive and equals $\frac{3}{2} {2i \choose i}.$
\end{theorem}

\begin{proof}
	Again, we must have for a nonzero summand that $2n \ge k+1 = 2i+1$ but that $(3i+1) - n -1 \ge k-1,$ so $n \ge i+1, n \le i+1$ or $n = i+1.$ Thus, $$J(3i+2, 2i) = I(i+1, 2i) {2i-1 \choose 2i-1} = I(i+1, 2i).$$
  
  Now, we have the summands in $I(i+1, 2i)$ are positive only if $i+1 \ge 2i - j$ and $i+1 \ge j+1.$ Thus, $i \ge j \ge i-1.$ In other words, 
\begin{align*}
	I(i+1, 2i) &= \frac{1}{i+1} {i+1 \choose i+1} {i+1 \choose i}{i+(i-1) \choose i} + \frac{1}{i+1}{i+1 \choose i}{i+1 \choose i+1} {i+i \choose i} \\
  &= {2i-1 \choose i} + {2i \choose i} \\
  &= \frac{3}{2} {2i \choose i}
\end{align*}
as desired.
\end{proof}

Finally, we also look at the case of $\ell \equiv 0 \Mod 3,$ and determine $J(3i, 2i-1),$ since $J(3i, 2i-1) \neq 0$ but $J(3i, 2i) = 0$ by Proposition \ref{prop27}.

\begin{theorem}
    We have that $J(3i, 2i-1) = 4$ if $i = 1$ and $J(3i, 2i-1) = {2i-2 \choose i} \cdot \frac{9i^2+2i-4}{2(i-1)}$ if $i \ge 2$.
\end{theorem}

\begin{proof}
    For $i = 1,$ it can be checked from \cite[Table 2]{BeanTannockUlfarsson2016}, so we assume $i \ge 2$. Again, for $I(n, k)$ to be positive, we must have $2n \ge k+1,$ so
\[J(3i, 2i-1) = \sum\limits_{n = i}^{3i} I(n, 2i-1) {3i-n-1 \choose 2i-2}.\]
    As ${3i-n-1 \choose 2i-2}$ is only positive when $3i-n-1 \ge 2i-2,$ or equivalently when $n \le i+1,$ we get
\[J(3i, 2i-1) = \sum\limits_{n = i}^{i+1} I(n, 2i-1) {3i-n-1 \choose 2i-2} = (2i-1) I(i, 2i-1) + I(i+1, 2i-1).\]
    Now, the summation for $I(i, 2i-1)$ is positive only when $i \ge (2i-1)-j$ and $i \ge j+1,$ so we must have $j = i-1.$ The summation for $I(i+1, 2i-1)$ is positive only when $i+1 \ge (2i-1)-j$ and $i+1 \ge j+1,$ so we must have $j \in \{i-2, i-1, i\}$. Therefore, we have
\begin{align*}
    J(3i, 2i-1) &= (2i-1) I(i, 2i-1) + I(i+1, 2i-1) \\
    &= \frac{2i-1}{i} {2i-2 \choose i-1} + \frac{1}{i+1} \left({i+1 \choose i-1} {2i-2 \choose i} + (i+1)^2 {2i-1 \choose i} + {i+1 \choose i-1} {2i \choose i}\right)
\end{align*}
    Using simple formulas to compute the ratios of ${2i-2 \choose i-1}, {2i-1 \choose i},$ and ${2i \choose i}$ to ${2i-2 \choose i},$ we can write this as
\begin{align*}
    &\hspace{0.5cm} {2i-2 \choose i} \left(\frac{2i-1}{i} \cdot \frac{i}{i-1} + \frac{i}{2} + (i+1) \cdot \frac{(2i-1)}{i-1} + \frac{i}{2} \cdot \frac{2(2i-1)}{i-1} \right) \\
    &= {2i-2 \choose i} \left(\frac{2i-1}{i-1} + \frac{i}{2} + \frac{(2i-1)(i+1)}{i-1} + \frac{i(2i-1)}{i-1} \right) \\
    &= {2i-2 \choose i} \left(\frac{9i^2+3i-4}{2(i-1)}\right).
\end{align*}
    We note that this formula matches \cite[Table 2]{BeanTannockUlfarsson2016} for $i = 2, 3, 4.$
\end{proof}

\begin{remark}
    Note that our above proof indeed requires $i \ge 2,$ since we sum over $j \in \{i-2, i-1, i\}$ and use ratios that do not make sense if $i = 1.$
\end{remark}

\section{Second Conjecture}

First, we note that left-to-right minima and right-to-left maxima are especially important in the study of $123$-avoiding permutations for the following reason:

\begin{proposition} \label{prop:123}
	A permutation $\pi \in S_n$ is $123$-avoiding if and only if $\pi_i$ is a left-to-right minimum or a right-to-left maximum for every $1 \le i \le k$.
\end{proposition}

\begin{proof}
	If $\pi$ is $123$-avoiding and $\pi_i$ is not a left-to-right minimum, there exists $j < i$ such that $\pi_j < \pi_i$. Also, if $\pi_i$ is not a a right-to-left maximum, there exists $k > i$ such that $\pi_k > \pi_i$. Thus, $i < j < k$ and $\pi_i < \pi_j < \pi_k,$ so $\pi$ contains $123$.
	
	To prove the converse, if $\pi$ contains $123,$ then there exist $i < j < k$ such that $\pi_i < \pi_j < \pi_k.$ Clearly, $\pi_j$ not a left-to-right minimum as $\pi_i < \pi_j$ and $i < j,$ and $\pi_j$ is not a right-to-left maximum as $\pi_j < \pi_k$ and $j < k$.
\end{proof}

Next, similarly to the staircase grid, we now consider the related \emph{boundary grid}, specifically for $123$-avoiding permutations.

\begin{definition}
	Given a $123$-avoiding permutation $\pi$ of length $n$, define the \emph{boundary grid} of $\pi$ as a subset of an $(n-1) \times (n-1)$ grid of boxes such that if we label the vertices of the boxes from $(1, 1)$ to $(n, n)$ in a standard Cartesian coordinate way, a box is in the boundary grid if and only if it is above and to the right of at least one left-to-right minimum and below and to the left of at least one right-to-left maximum. See Figure \ref{DowncoreGraph} for an example.
\end{definition}

We also define a \textit{Young diagram} and a \textit{skew Young diagram}. We remark that all boundary grids are skew Young diagrams.

\begin{definition}
    A \textit{Young diagram} is a finite, connected collection of boxes such that there exists a bottom-left most box that we label as $(1, 1)$. Moreover, if the box $(k, \ell)$ (which is in the $k^{\text{th}}$ column from left and $\ell^{\text{th}}$ row from bottom) in the Young diagram, then every box $(k', \ell')$ with $1 \le k' \le k, 1 \le \ell' \le \ell$ is also in the diagram.
\end{definition}

\begin{definition}
    A \textit{skew Young diagram} is a finite collection of boxes such that there exists a Young diagram $\lambda$ and another Young diagram $\mu \subset \lambda$ with the same bottom-left most box, such that a box is in the skew Young diagram if and only if it is in $\lambda$ but not in $\mu$.
\end{definition}

\begin{remark}
    When dealing with a boundary grid of a permutation $\pi$ or any skew Young diagram generally, we will always use the convention of $(x, y)$ to denote the box in the $x^{\text{th}}$ column from left and $y^{\text{th}}$ row from bottom. This way, we correspond with the location of $(i, \pi_i)$.
\end{remark}

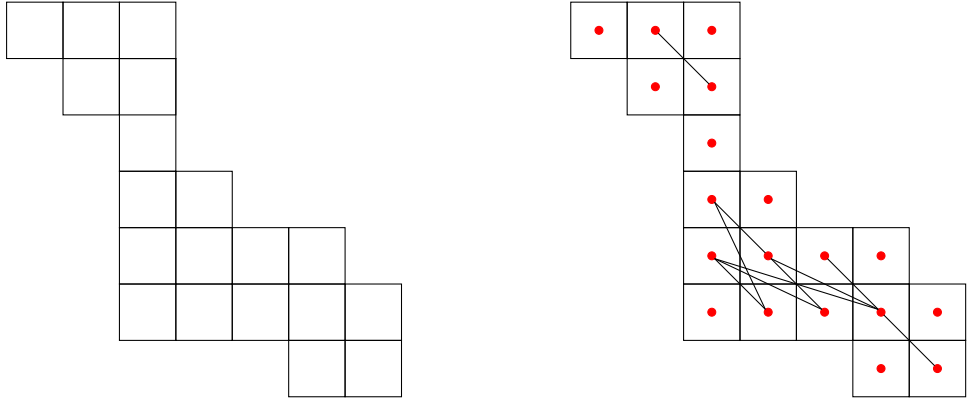
\begin{figure}
\centering
\begin{tikzpicture}[scale=0.75]
\foreach \i in {1,...,3} {
  \draw (\i - 10, 7) rectangle (\i - 9, 8);
  \draw (\i, 7) rectangle (\i + 1, 8);
  \filldraw [red] (\i + 0.5, 7.5) circle (2pt);
}
\foreach \i in {2,3} {
  \draw (\i - 10, 6) rectangle (\i - 9, 7);
  \draw (\i, 6) rectangle (\i + 1, 7);
  \filldraw [red] (\i + 0.5, 6.5) circle (2pt);
}
\foreach \i in {3} {
  \draw (\i - 10, 5) rectangle (\i - 9, 7);
  \draw (\i, 5) rectangle (\i + 1, 6);
  \filldraw [red] (\i + 0.5, 5.5) circle (2pt);
}
\foreach \i in {3, 4} {
  \draw (\i - 10, 4) rectangle (\i - 9, 5);
  \draw (\i, 4) rectangle (\i + 1, 5);
  \filldraw [red] (\i + 0.5, 4.5) circle (2pt);
}
\foreach \i in {3,...,6} {
  \draw (\i - 10, 3) rectangle (\i - 9, 4);
  \draw (\i, 3) rectangle (\i + 1, 4);
  \filldraw [red] (\i + 0.5, 3.5) circle (2pt);
}
\foreach \i in {3,...,7} {
  \draw (\i - 10, 2) rectangle (\i - 9, 3);
  \draw (\i, 2) rectangle (\i + 1, 3);
  \filldraw [red] (\i + 0.5, 2.5) circle (2pt);
}
\foreach \i in {6,7} {
  \draw (\i - 10, 1) rectangle (\i - 9, 2);
  \draw (\i, 1) rectangle (\i + 1, 2);
  \filldraw [red] (\i + 0.5, 1.5) circle (2pt);
}

\draw (2.55, 7.45) -- (3.45, 6.55);
\draw (3.55, 4.45) -- (4.45, 3.55);
\draw (3.55, 4.45) -- (4.45, 2.55);
\draw (3.55, 3.45) -- (4.45, 2.55);
\draw (3.55, 3.45) -- (5.45, 2.55);
\draw (3.55, 3.45) -- (6.45, 2.55);
\draw (4.55, 3.45) -- (5.45, 2.55);
\draw (4.55, 3.45) -- (6.45, 2.55);
\draw (5.55, 3.45) -- (6.45, 2.55);
\draw (6.55, 2.45) -- (7.45, 1.55);
\end{tikzpicture}
\caption{On the left, we see the boundary grid for the permutation $\pi = 76285143$. On the right, we see the downcore graph for the same permutation $\pi$, with the vertices as the centers of each box.}
\label{DowncoreGraph}
\end{figure}

	Note that the boundary grid of a $123$-avoiding permutation isn't always connected. In fact, one can see that the boundary grid is not connected if and only if the permutation $\pi$ can be written as $\pi_1 \pi_2 \cdots \pi_k \pi_{k+1} \cdots \pi_n$ for some $1 \le k \le n-1$ such that $\pi_i > \pi_j$ for all $i \le k, j > k$ \cite{BeanTannockUlfarsson2016}. In this case, we call the permutation \textit{skew-decomposable}, and otherwise call the permutation \textit{skew-indecomposable}. One can also see that any connected boundary grid (or any connected component of a boundary grid) has the shape of a skew Young diagram.
  
  Similar to the downcore on the staircase grid, we can define the \emph{downcore graph} on any skew Young diagram, and therefore on any boundary grid.
  
\begin{definition}
	The \emph{downcore graph} on a skew Young diagram is a graph with the boxes of the diagram as vertices. If $(x, y)$ denotes the box in the $x^{\text{th}}$ column from left and $y^{\text{th}}$ row from bottom (to correspond with the Cartesian coordinates), there is an edge between $(i, j)$ and $(k, \ell)$ if the following two conditions are met:
  
\begin{enumerate}
	\item $i < k, j > \ell$ or $i > k, j < \ell$
	\item $(i, \ell), (k, j)$ are both in the skew Young diagram.
\end{enumerate}

	The second condition is equivalent to the rectangle created with corner boxes $(i, j), (i, \ell), (k, j),$ and $(k, \ell)$ being entirely contained within the boundary grid. See Figure \ref{DowncoreGraph} for an example of a downcore graph on a boundary grid.
\end{definition}

\begin{definition}
	We define a graph to be \emph{pure} if every maximal independent set has the same size, i.e. the size of every maximal clique in the complement graph is the same. We remark that this is closely related to the definition of a pure simplicial complex.
\end{definition}

Bean et al. make the following conjecture.

\begin{conjecture} \label{Conjecture2}
	The downcore graph of the boundary grid of a $123$-avoiding permutation $\pi$ is pure if and only if $\pi$ avoids $2143$.
\end{conjecture}

	While they do not fully address the ``if'' direction, they prove most of the main ideas for this direction, so this section is primarily focused on the ``only if'' direction.

\begin{lemma} \label{lem:Row}
	Suppose we have a fixed maximal independent set in the downcore graph of a skew Young diagram. Then, the independent set must contain at least one box in every row.
\end{lemma}

\begin{proof}
    Assume for a general row $x$ that the first box in the row is in column $a_x$ and the last box is in column $b_x$. Now, fix a row $h$. Then, if $a_h < a_{h-1}$ or if $h = 1$, the box at position $(a_h, h)$ is a lower-left corner and thus has degree $0$ in the downcore. Similarly, if $b_h > b_{h+1}$ or if $h = n-1$ (since the boundary grid has $n-1$ rows and $n-1$ columns), the box at position $(b_h, h)$ is an upper-right corner and thus has degree $0$ in the downcore. Therefore, it suffices to prove the lemma for row $h$ when $a_h = a_{h-1}, b_h = b_{h+1}.$
  
    Next, choose the largest $k > h$ such that $b_k = b_h.$ Consider the rectangle spanning rows $h+1$ to $k$ and columns $a_h$ to $b_h.$ If a maximal independent set does not contain boxes in this rectangle, then the set must contain a box in row $h$. Otherwise, we can add $(b_h, h)$ to our independent set, as $(b_h, h)$ only shares edges in the downcore with boxes in the rectangle spanning rows $h+1$ to $k$ and columns $a_h$ to $b_h-1$. 
  
    Now, if the independent set contains a box in this rectangle, consider the smallest $\ell$ such that $h+1 \le \ell \le k$ and the independent set contains a box in row $\ell$ of the rectangle. Consider the leftmost box in row $\ell$ which is in the independent set and say it is in column $c$ ($a_h \le c \le b_h$). We show that $(c, h)$ does not share any edges with the independent set.
  
    Suppose that $(c, h)$ and $(c', h')$ share an edge for some $(c', h')$. If $c > c'$ and $h < h'$, then clearly $\ell < h'$ also. But then $(c', h')$ and $(c, \ell)$ share an edge, so they can't both be in the independent set. Else, suppose $c < c'$ and $h > h'$. But then $(c, h)$ and $(c', h')$ sharing an edge means that $(c, h')$ is a box in the skew Young diagram, and $(c', \ell)$ is also in the skew Young diagram as $c' \le b_h$. This means $(c, \ell)$ and $(c', h')$ share an edge, so $(c', h')$ is not in the independent set. Therefore, row $h$ must contain at least one box in any maximal independent set, or else we can add $(c, h)$ to the maximal independent set.
\end{proof}

\begin{lemma} \label{lem:Duplication}
	Suppose that a certain skew Young diagram has a downcore graph that is not pure and one of its rows is duplicated, i.e. a horizontal line is drawn through the middle of a row and makes each box in that row two boxes. This new skew Young diagram's downcore graph is also not pure.
\end{lemma}

\begin{proof}
	Choose a maximal independent set for the original grid. For the new diagram, choose a set of boxes so that the bottom duplicated row and all unduplicated rows have same boxes as in the original independent set. For the top duplicated row, only choose the rightmost box in the row that was in the original independent set to be in our new set. (See Figure \ref{Duplication} for an example). This set of boxes forms an independent set in the downcore. To see why, there are clearly no edges in the downcore between two boxes of our new set in the two duplicated rows. Moreover, any edge between any two chosen boxes that aren't both in the duplicated rows must have shared an edge in the original diagram's downcore.
	
	This set of boxes is also a maximal independent set. Suppose that another box on the grid does not share any edges with any of the boxes in our set. Consider the location of the box on the original grid (if the box were on one of the duplicated rows, choose it to be the relative box on the original row). Note that if the location on the original grid coincided with one of the boxes, then on the duplicated grid there would be two pairs of boxes in the set that would make a rectangle, which clearly has an edge in the downcore. 
	
	Otherwise, if this box did not coincide with one of the other boxes when brought back to the original grid, it is clear that if this box shares an edge in the new downcore, it must share the same edge in the downcore if we remove the top duplicated row. Thus, our set of boxes forms a maximal independent set. To finish, note that if we have two maximal independent sets of different sizes in our original downcore, the maximal independent sets we form in the new downcore each increase by $1$, so they have different sizes.
\end{proof}

\begin{figure}
\centering
\begin{tikzpicture}[scale=0.75]

\draw[fill=red] (1, 4) rectangle (2, 5);
\draw (2, 4) rectangle (3, 5);
\draw (3, 4) rectangle (4, 5);
\draw[fill=red] (4, 4) rectangle (5, 5);
\draw[fill=red] (2, 3) rectangle (3, 4);
\draw[fill=red] (3, 3) rectangle (4, 4);
\draw[fill=red] (4, 3) rectangle (5, 4);
\draw[fill=red] (3, 2) rectangle (4, 3);
\draw (4, 2) rectangle (5, 3);
\draw[fill=red] (5, 2) rectangle (6, 3);
\draw[fill=red] (4, 1) rectangle (5, 2);
\draw[fill=red] (5, 1) rectangle (6, 2);
\draw[thick, dashed] (3, 2.5) -- (6, 2.5);

\draw[fill=red] (9, 4) rectangle (10, 5);
\draw (10, 4) rectangle (11, 5);
\draw (11, 4) rectangle (12, 5);
\draw[fill=red] (12, 4) rectangle (13, 5);
\draw[fill=red] (10, 3) rectangle (11, 4);
\draw[fill=red] (11, 3) rectangle (12, 4);
\draw[fill=red] (12, 3) rectangle (13, 4);
\draw (11, 2) rectangle (12, 3);
\draw (12, 2) rectangle (13, 3);
\draw[fill=red] (13, 2) rectangle (14, 3);
\draw[fill=red] (11, 1) rectangle (12, 2);
\draw (12, 1) rectangle (13, 2);
\draw[fill=red] (13, 1) rectangle (14, 2);
\draw[fill=red] (12, 0) rectangle (13, 1);
\draw[fill=red] (13, 0) rectangle (14, 1);

\draw[fill=red] (1, 10) rectangle (2, 11);
\draw (2, 10) rectangle (3, 11);
\draw (3, 10) rectangle (4, 11);
\draw[fill=red] (4, 10) rectangle (5, 11);
\draw[fill=red] (2, 9) rectangle (3, 10);
\draw (3, 9) rectangle (4, 10);
\draw[fill=red] (4, 9) rectangle (5, 10);
\draw[fill=red] (3, 8) rectangle (4, 9);
\draw[fill=red] (4, 8) rectangle (5, 9);
\draw[fill=red] (5, 8) rectangle (6, 9);
\draw[fill=red] (4, 7) rectangle (5, 8);
\draw (5, 7) rectangle (6, 8);
\draw[thick, dashed] (3, 8.5) -- (6, 8.5);

\draw[fill=red] (9, 10) rectangle (10, 11);
\draw (10, 10) rectangle (11, 11);
\draw (11, 10) rectangle (12, 11);
\draw[fill=red] (12, 10) rectangle (13, 11);
\draw[fill=red] (10, 9) rectangle (11, 10);
\draw (11, 9) rectangle (12, 10);
\draw[fill=red] (12, 9) rectangle (13, 10);
\draw (11, 8) rectangle (12, 9);
\draw (12, 8) rectangle (13, 9);
\draw[fill=red] (13, 8) rectangle (14, 9);
\draw[fill=red] (11, 7) rectangle (12, 8);
\draw[fill=red] (12, 7) rectangle (13, 8);
\draw[fill=red] (13, 7) rectangle (14, 8);
\draw[fill=red] (12, 6) rectangle (13, 7);
\draw (13, 6) rectangle (14, 7);

\end{tikzpicture}

\caption{The above figures are meant to help explain Lemma \ref{lem:Duplication}. The top-left and bottom-left diagrams are skew Young diagrams where the red shaded boxes form maximal independent sets of different sizes. Then, if we duplicate the second row from the bottom, we can create two maximal independent sets of differing sizes for the new skew Young diagram. We do this by choosing the red boxes in the bottom duplicate row identically to the original row, and choosing only the rightmost red box in the original row to be red in the top duplicate row.}
\label{Duplication}
\end{figure}
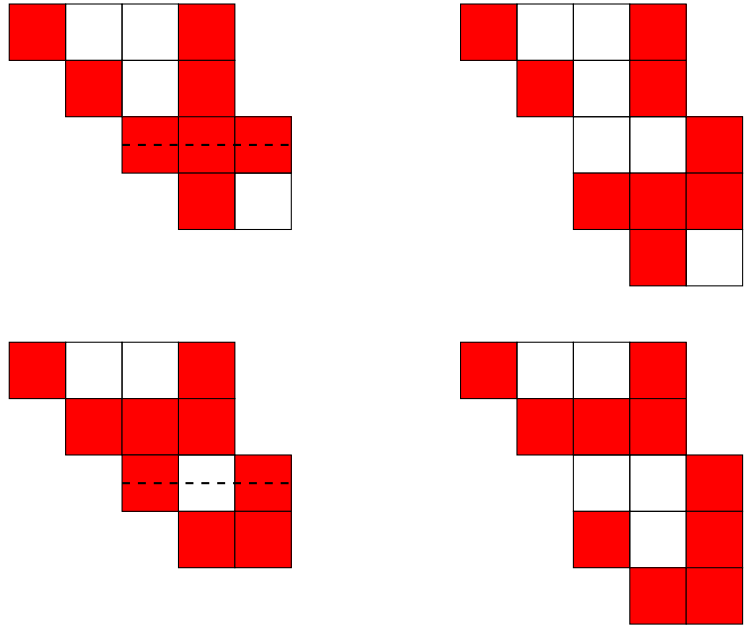

%
%
%

\begin{lemma} \label{lem:Construction}
	For $n \ge 4,$ consider the boundary grid created by the permutation $(n-2)(n-3)\cdots(1)(n)(n-1)$. This boundary grid is not pure.
\end{lemma}

\begin{proof}
	Suppose first that $n = 4$. Then, the shaded boxes in Figure \ref{4figure} represent maximal independent sets (this is straightforward to verify). But the left diagram has $6$ selected boxes and the right diagram has $5$ selected boxes, so the downcore graph is not pure. 

\begin{figure}
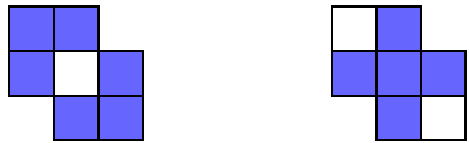
 
\begin{ytableau}
\none & *(blue!60) & *(blue!60)\\
\none & *(blue!60) & & *(blue!60) \\
\none & \none & *(blue!60) & *(blue!60) \\
\end{ytableau}
\begin{ytableau}
\none & \none & \none & \none & & *(blue!60) \\
\none & \none & \none & \none & *(blue!60) & *(blue!60) & *(blue!60) \\
\none & \none & \none & \none & \none & *(blue!60) &\\
\end{ytableau}

\caption{Diagram for $n = 4$, corresponding to Lemma \ref{lem:Construction}.}
\label{4figure}
\end{figure}

	Now, suppose that $n \ge 5$. Let $S_n$ be the set of boxes on the perimeter. The set $S_n$ is an independent set since the $n-2$ boxes on the bottom-left staircase (shaded green in Figure \ref{GeneralFigure}) have degree $0$ in the downcore. It is also obvious that the $n-2$ boxes in the topmost row and the $n-2$ boxes in the rightmost column do not share any edges between them (as the top-right corner is missing) and do not share any edges among them (as the top boxes are in the same row and the right boxes are in the same column). The set $S_n$ is also a maximal independent set since for any unshaded box in the $k^{\text{th}}$ row from the bottom (where $2 \le k \le n-2$), the box labeled $(n-1-k, n-1)$ shares an edge with it.

	Now, consider the set $T_n$ obtained by removing $(n-3, n-1),$ the second rightmost box of the top row, and $(n-1, 1)$, the rightmost box of the bottom row, and adding $(n-2, 2)$, the middle box of the second bottom row (as shown in yellow in the right half of Figure \ref{GeneralFigure}). To see why $T_n$ is an independent set, note that the only box we have to check is $(n-2, 2)$. However, the only box $(n-2, 2)$ shares an edge with in the top row is the box $(n-3, n-1)$, which is not in our set. Also, every box in the rightmost column is in the same row or in a higher row as $(n-2, 2)$, except the box labeled $(n-1, 1)$, which is not in the set. Finally, as the boxes in the bottom-left staircase all have degree $0$, $T_n$ is an independent set.
	
	To show the set is maximal, note that the two deleted boxes each share an edge with the added box, so they cannot be added. The remaining boxes not in the set are in rows $3$ to $n-2$ and thus share edges with some box in the top row and columns $1$ to $n-4$, which are all in the set. Therefore, the set is maximal.
  
\begin{figure}
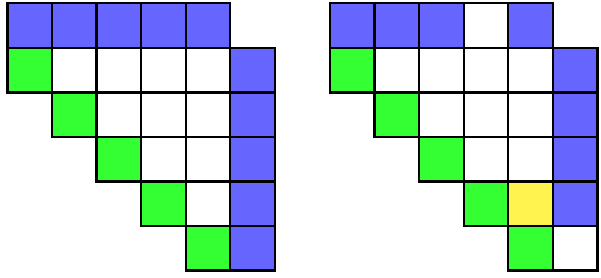


\begin{ytableau}
\none & *(blue!60) & *(blue!60) & *(blue!60) & *(blue!60) & *(blue!60)\\
\none & *(green!80) & & & & & *(blue!60)\\
\none & \none & *(green!80) & & & & *(blue!60)\\
\none & \none & \none & *(green!80) & & & *(blue!60) \\
\none & \none & \none & \none & *(green!80) & & *(blue!60) \\
\none & \none & \none & \none & \none & *(green!80) & *(blue!60) \\
\end{ytableau}
\begin{ytableau}
\none & *(blue!60) & *(blue!60) & *(blue!60) & & *(blue!60)\\
\none & *(green!80) & & & & & *(blue!60)\\
\none & \none & *(green!80) & & & & *(blue!60)\\
\none & \none & \none & *(green!80) & & & *(blue!60) \\
\none & \none & \none & \none & *(green!80) & *(yellow!80) & *(blue!60) \\
\none & \none & \none & \none & \none & *(green!80) & \\
\end{ytableau}

\caption{General diagram for $n \ge 5$, corresponding to Lemma \ref{lem:Construction}. In this example we have set $n = 7$.}
\label{GeneralFigure}
\end{figure}
  
  It is clear that there are $3(n-2)$ boxes in the maximal set in the left diagram and $3(n-2)-1$ boxes in the maximal set in the right diagram. Thus, for $n \ge 5$, the boundary grid's downcore graph is not pure.
\end{proof}

\begin{theorem} \label{thm:Conj6.1Hard}
	The downcore of the boundary grid of a $123$-avoiding permutation is pure only if the permutation avoids $2143$.
\end{theorem}

\begin{proof} 
	Suppose otherwise. Consider the smallest $n$ such that a permutation in $Av_n(123) \cap Co_n(2143)$ has a boundary grid with a pure downcore graph. Call the permutation $\pi$. If $\pi$ is skew-decomposable, then $\pi$'s boundary grid can be split into two disjoint boundary grids, both of which must be pure and one of which must correspond to a permutation containing $2143$, contradicting minimality of $n$. Therefore, $\pi$ is skew-indecomposable, so $\pi_1$ is only a left-to-right minimum. If $\pi_2$ is a right-to-left maximum, then it clearly cannot be part of an occurrence of a $2143$ pattern. This also means $\pi_2 = n$ and $\pi_1 \neq n-1$, as $\pi$ is skew-indecomposable. In this case, the top row of the boundary grid only has one box (in the leftmost column), which has degree $0$ in the downcore graph. Moreover, removing the top row will give us the boundary grid of $\pi_1 \pi_3 \cdots \pi_n$ but with the first column duplicated. Therefore, if the original downcore is pure, removing the box in the top row will still result in a pure downcore (as the box has degree $0$ in the downcore), and by Lemma \ref{lem:Duplication}, removing the first column of the boundary grid (which is identical to the second column) will still result in a pure downcore graph. This contradicts $n$'s minimality since the boundary grid for $\pi_1 \pi_3 \cdots \pi_n$ has a pure downcore graph but $\pi_1 \pi_3 \cdots \pi_n$ contains a $2143$ pattern, so $\pi_2$ must be a left-to-right minimum.
  
    Now, assume that $\pi_1, \ldots , \pi_k$ are left-to-right minima ($k \ge 2$) and $\pi_{k+1}$ is a right-to-left maximum. Also, suppose that $\pi_{k+\ell}$ is the next right-to-left maximum (there must be at least two right-to-left maxima or else the permutation would avoid $2143$). If $\pi_{k+\ell}<\pi_k,$ then $\pi_i < \pi_k$ for all $i > k,$ which contradicts skew-indecomposability.
  
    If $\pi_{k-1} > \pi_{k+\ell} > \pi_k,$ then $\pi_i = n-i$ for $1 \le i \le k-1,$ $\pi_k = n,$ and $\pi_{k+\ell} = n-k.$ This means the $k^{\text{th}}$ row from top has exactly one box in the boundary grid, at column $k$. As a result, the box at $(k, n-k)$ has degree $0$ in the downcore, and the boxes above row $k$ do not share any edges with the boxes below row $k$. Thus, if downcore of $\pi$ is pure, the downcore of the rows below row $k$ are pure. However, note that for all $1 \le i \le k-1$, $\pi_{k+1} = n$ is the only element of the form $\pi_r$ for some $r > i$ such that $\pi_{r} > \pi_i$. Thus, none of $\pi_1, \dots, \pi_{k-1},\pi_{k+1}$ can be the first element of a $2143$ classical pattern. Therefore, any $2143$ pattern must begin at least at $\pi_k$. But since $\pi_{k+1}$ must be the largest element in the $2143$ pattern, it cannot be in a $2143$ classical pattern, so the reduction of $\pi_k\pi_{k+2}\cdots\pi_n$ must contain a $2143$. However, the boundary grid of $\pi$ restricted to rows below row $k$ is equivalent to the boundary grid of $\pi_k\pi_{k+2}\cdots\pi_n$ with the first column duplicated, meaning that the downcore of the boundary grid of $\pi_k\pi_{k+2}\cdots\pi_n$ must be pure, a contradiction to $n$'s minimality.
  
    Now, suppose that $\pi_{i-1}>\pi_{k+\ell}>\pi_i,$ where $1 \le i \le k-1$ and $\pi_0 := \pi_{k+1}$. If $i > 1,$ then it is easy to see that the leftmost column contains only one box (in the top row) which has degree $0$ in the downcore. As a result it can be eliminated, i.e. it suffices to prove that the boundary grid after removing the box in the top-left corner is not pure. But then the top two rows are identical, so Lemma \ref{lem:Duplication} tells us that we can remove one of the top two rows. This new diagram is equivalent to the boundary grid of $\pi_2\cdots\pi_n$ which means to prevent contradicting minimality, we can assume $i = 1,$ so $\pi_{k+1} = n, \pi_{k+\ell} = n-1.$
  
    Next, note that all elements to the right of $\pi_{k+1}$ are at most $n-1 = \pi_{k+\ell}$. This means that the first $k+1$ columns of the boundary grid of $\pi$ is equivalent to the the boundary grid of $(k)(k-1) \cdots (1)(k+2)(k+1)$ except with perhaps some of the rows replicated. Thus, by using the methods of Lemma \ref{lem:Duplication} and Lemma \ref{lem:Construction}, we can create two maximal independent sets for the first $k+1$ columns, the second maximal independent set having one less box.

\begin{figure}
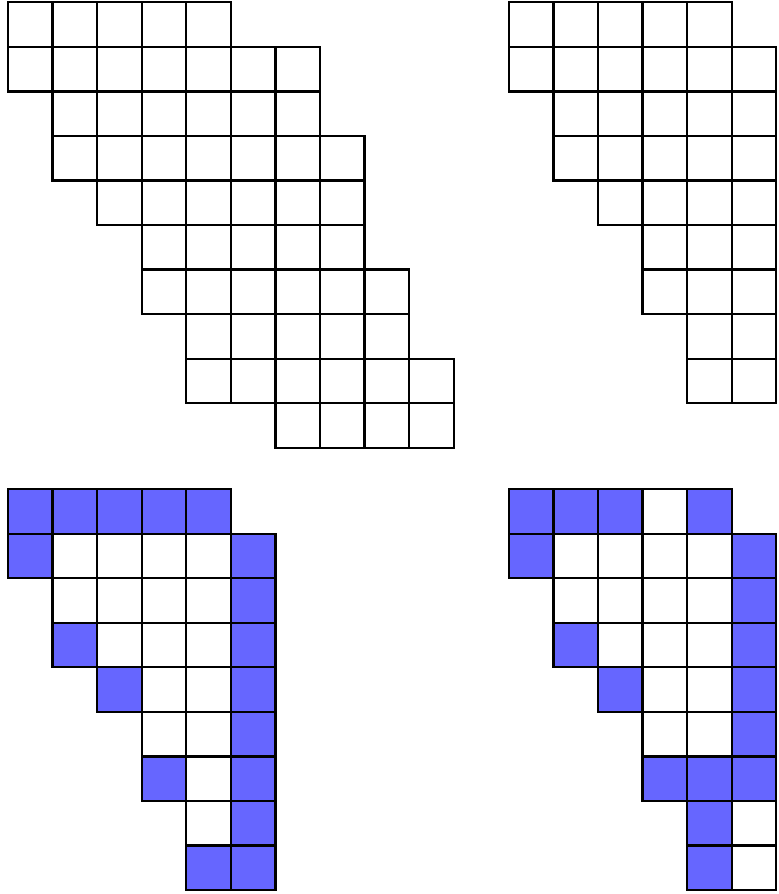

\begin{ytableau}
 *(white!60) & & & & \\
 *(white!60) & & & & & &\\
 \none & & & & & & \\
 \none & & & & & & & \\
 \none & \none & & & & & & \\
 \none & \none & \none & & & & & \\
 \none & \none & \none & & & & & & \\
 \none & \none & \none & \none & & & & & \\
 \none & \none & \none & \none & & & & & & \\
 \none & \none & \none & \none & \none & \none & & & & \\
\end{ytableau}
\begin{ytableau}
\none & *(white!60) & & & & \\
\none & *(white!60) & & & & &\\
\none & \none & & & & & \\
\none & \none & & & & & \\
\none & \none & \none & & & & \\
\none & \none & \none & \none & & & \\
\none & \none & \none & \none & & & \\
\none & \none & \none & \none & \none & & \\
\none & \none & \none & \none & \none & & \\
\end{ytableau}

\vspace{.5cm}

\begin{ytableau}
*(blue!60) & *(blue!60) & *(blue!60) & *(blue!60) & *(blue!60)\\
*(blue!60) & & & & & *(blue!60)\\
\none & & & & & *(blue!60)\\
\none & *(blue!60) & & & & *(blue!60)\\
\none & \none & *(blue!60) & & & *(blue!60) \\
\none & \none & \none & & & *(blue!60) \\
\none & \none & \none & *(blue!60) & & *(blue!60) \\
\none & \none & \none & \none & & *(blue!60) \\
\none & \none & \none & \none & *(blue!60) & *(blue!60) \\
\end{ytableau}
\begin{ytableau}
\none & \none & \none & \none & \none & *(blue!60) & *(blue!60) & *(blue!60) & & *(blue!60)\\
\none & \none & \none & \none & \none & *(blue!60) & & & & & *(blue!60)\\
\none & \none & \none & \none & \none & \none & & & & & *(blue!60)\\
\none & \none & \none & \none & \none & \none & *(blue!60) & & & & *(blue!60)\\
\none & \none & \none & \none & \none & \none & \none & *(blue!60) & & & *(blue!60) \\
\none & \none & \none & \none & \none & \none & \none & \none & & & *(blue!60) \\
\none & \none & \none & \none & \none & \none & \none & \none & *(blue!60) & *(blue!60) & *(blue!60) \\
\none & \none & \none & \none & \none & \none & \none & \none & \none & *(blue!60) & \\
\none & \none & \none & \none & \none & \none & \none & \none & \none & *(blue!60) & \\
\end{ytableau}

\caption{We give an example for the case $\pi_1 > \pi_{k+1} > \pi_{k+\ell}.$ Here, $n = 11, k = 5,$ and $\ell = 3$. The top-left figure is the boundary grid of a permutation $\pi = (9, 7, 6, 4, 2, 11, 1, 10, 8, 5, 3)$. The top-right figure is the first $k+1 = 6$ columns of the boundary grid, which is equivalent to the grid in Figure \ref{GeneralFigure} with the $3^{\text{rd}}, 5^{\text{th}},$ and $6^{\text{th}}$ rows duplicated. We can thus use our two maximal independent sets of Figure \ref{GeneralFigure} and Lemma \ref{lem:Duplication} to get two maximal independent sets with different sizes, as seen by the bottom-left and bottom-right figures. Therefore, the downcore graph of the top-left figure is not pure, since we can choose boxes for the left $6$ columns in either way, but a box not in the left $6$ columns shares an edge with a shaded box if the left $6$ columns are shaded the first way, if and only if the same is true if the left $6$ columns are shaded the second way.}
\label{MainTheorem}
\end{figure}

	Now, if there are any additional columns, select boxes from these columns so that when combined with the second (right) selection of boxes for the first $k+1$ columns, the boxes form a maximal independent set for the entire boundary grid. We show that the same selection, when combined with the first selection of boxes for the first $k+1$ columns, forms an independent set for the entire boundary grid.

	Note that the only way for this selection to not form an independent set is for one of the boxes in the first set to share an edge with one of the boxes not in the first $k+1$ rows but for none of the boxes in the second set to share an edge with the same box.

	The only boxes in the first set but not in the second set which can actually have an edge with any box beyond the first $k+1$ columns must be in the $(k+1)^{\text{th}}$ column. But then if $(k+1, a)$ and $(\ell, b)$ are connected where $\ell > k+1,$ it is clear that $(k, a)$ and $(\ell, b)$ do as well, since $(k+1, b)$ being in the grid means $(k, b)$ is also in the grid. It is easy to see from Figure \ref{MainTheorem} that if $(k+1, a)$ is in the first set but not in the second set, then $(k, a)$ is in the second set. This completes our proof.
\end{proof}

We now continue to prove the ``if'' direction.

Bean et al. \cite[Lemma 5.3]{BeanTannockUlfarsson2016} proved that the boundary grid of every skew-indecomposable $2143$-avoiding permutation can be decomposed into a series of staircase grids (with the diagonal going from top-left to bottom-right), consecutive staircases sharing exactly one box. As a result, no two boxes in different staircases can share an edge, and therefore it suffices to prove that the downcore graph of staircase grid is pure.

\begin{theorem} \label{thm:reverse}
	The downcore graph of a staircase grid with the diagonal going from top-left to bottom-right is pure, and if the size of the staircase is $n$, then every maximal independent set has size $2n-1$.
\end{theorem}

\begin{proof}
	We proceed by induction. Note for $n = 1$ and $n = 2$ this is trivial. Now, suppose that $n \ge 3.$ Also, assume without loss of generality that the staircase contains the bottom-left but not top-right corner, so that the boxes are of the form $(i, j)$, where $i, j \ge 1$ and $i+j \le n+1$.
  
  Next, choose a maximal independent set with the selected boxes in the left column $(1, a_1), \ldots, (1, a_k)$ where $a_1 < \cdots < a_k.$ Note that $a_1 = 1$ and $a_k = n$ because $(1, 1)$ and $(1, n)$ do not have any edges in the downcore graph. Now, any box with coordinates $(i, j)$ where $j < a_t$ and $i \le n+1-a_t$ for some $a_t$ cannot be included in the independent set, as it shares an edge with $(1, a_t).$ Therefore, if $(i, j)$ is in the independent set, $a_t > j \ge a_{t-1}$ and $i \ge n+2-a_t$ for some $2 \le t \le k.$ Therefore, the only allowable boxes in the independent set form a disjoint set of staircase grids of size $a_2-a_1, \dots, a_{k}-a_{k-1}$ (see Figure \ref{Converse} for an example). If $(i, j)$ and $(k, \ell)$ are in different staircase grids, then if we assume without loss of generality that $i > k,$ then $j < a_t \le \ell$ for some $t$, which means that $i > n+1-a_t.$ However, this means that $(i, j)$ and $(k, \ell)$ cannot share an edge in the downcore, since $i+\ell \ge i+a_t > n+1,$ which means that $(i, \ell)$ is not contained in the original grid.

\begin{figure}
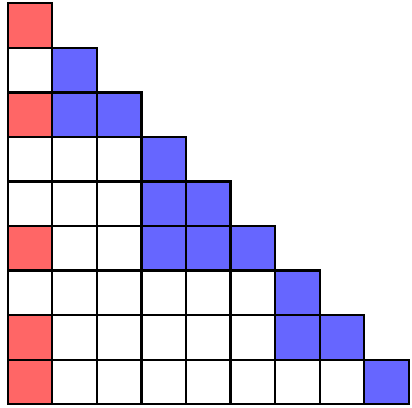

\begin{ytableau}
\none & *(red!60)\\
\none & *(white!60) & *(blue!60)\\
\none & *(red!60) & *(blue!60) & *(blue!60)\\
\none & *(white!60) & *(white!60) & *(white!60) & *(blue!60)\\
\none & *(white!60) & *(white!60) & *(white!60) & *(blue!60) & *(blue!60)\\
\none & *(red!60) & *(white!60) & *(white!60) & *(blue!60) & *(blue!60) & *(blue!60)\\
\none & *(white!60) & *(white!60) & *(white!60) & *(white!60) & *(white!60) & *(white!60) & *(blue!60)\\
\none & *(red!60) & *(white!60) & *(white!60) & *(white!60) & *(white!60) & *(white!60) & *(blue!60) & *(blue!60)\\
\none & *(red!60) & *(white!60) & *(white!60) & *(white!60) & *(white!60) & *(white!60) & *(white!60) & *(white!60) & *(blue!60)\\
\end{ytableau}

\caption{Suppose, for example, that $n = 9$ and $a_1 = 1, a_2 = 2, a_3 = 4, a_4 = 7,$ and $a_5 = 9.$ Then, the blue boxes are the only other boxes not in column $1$ that do not share an edge with any red boxes in the downcore graph. Therefore, any maximal independent set containing all red boxes and no other boxes in the first column must have size $17$.}
\label{Converse}
\end{figure}

  Therefore, we are left with $k$ boxes in the first column and staircases of size $a_{t+1}-a_t$. Each of the $k$ smaller staircases can be filled to the maximum of $2(a_{t+1}-a_t)-1$ elements since they do not overlap with or share edges with any selected boxes in the first column or any boxes in the other smaller staircase grids. Then, by our induction hypothesis, any maximal independent set with $(1, a_1), \ldots, (1, a_k)$ chosen from the first column has size $$k + \sum\limits_{t = 1}^{k-1} \left(2(a_{t+1}-a_t)-1\right) = k + 2(n-1)-(k-1) = 1 + 2n-2 = 2n-1,$$ as desired.
\end{proof}

Combining Theorems \ref{thm:Conj6.1Hard} and \ref{thm:reverse} concludes our proof of Conjecture \ref{Conjecture2}.

West \cite{West1996} proved that $|Av_n(123, 2143)| = F_{2n-1},$ where $F_n$ is the Fibonacci sequence with $F_1 = F_2 = 1.$ As a result, we have the following:

\begin{corollary} \label{Counting}
	The number of (possibly skew-decomposable) permutations in $Av_n(123)$ with boundary grids with pure downcore graphs equals $F_{2n-1}.$
\end{corollary}

\newpage
\section*{Acknowledgments}
This research was funded by NSF grant 1358659 and NSA grant H98230-16-1-0026
as part of the 2016 Duluth Research Experience for Undergraduates (REU).

The author would like to thank Prof. Joe Gallian for running a wonderful REU and supervising the research, as well as for suggesting the problem.  The author would also like to thank Joe for helpful comments on editing the paper.

\bibliographystyle{plain}

\end{document}